\documentclass[11pt]{amsart}
\usepackage{amsmath}
\usepackage{amssymb}
\usepackage{amsthm}
\usepackage{graphics,graphicx}
\usepackage{array}
\usepackage{a4wide}
\usepackage{float}
\usepackage[linguistics]{forest}
\usepackage{algorithm,tikz}
\usetikzlibrary{shapes,snakes}
\usepackage{algpseudocode}

\newcommand{\nn}{{\mathbb N}}
\newcommand{\cc}{{\mathbb C}}

\theoremstyle{plain}
\newtheorem{theorem}{Theorem}
\newtheorem{proposition}{Proposition}

\theoremstyle{definition}

\newtheorem{example}[theorem]{Example}

\newcommand{\beq}{\begin{eqnarray*}}
	\newcommand{\feq}{\end{eqnarray*}}
\newcommand{\beqn}{\begin{eqnarray}}
	\newcommand{\feqn}{\end{eqnarray}}

\def\T{\mathcal{T}}
\def\L{\mathcal{L}}
\def\D{\mathcal{D}}
\def\M{\mathcal{M}}
\def \R {\mathcal{R}}
\begin{document}
	\title[Subregularity in infinitely labeled generating trees]{Subregularity in infinitely labeled generating trees of restricted permutations}
	\author[T. Mansour]{Toufik Mansour}
	\address{Department of Mathematics, University of Haifa,
		3498838 Haifa, Israel}
	\email{tmansour@univ.haifa.ac.il}
	\author[R. Rastegar]{Reza Rastegar}
	\address{Occidental Petroleum Corporation, Houston, TX 77046 and Departments of Mathematics and Engineering, University of Tulsa, OK 74104, USA}
	\email{reza\_rastegar2@oxy.com}
	\author[M. Shattuck]{Mark Shattuck}
	\address{Department of Mathematics, University of Tennessee,
		Knoxville, TN 37996, USA}
	\email{shattuck@math.utk.edu}

	\begin{abstract}
		In this paper, we revisit the application of generating trees to the pattern avoidance problem for permutations. In particular, we study this problem for certain general sets of patterns and propose a new procedure leveraging the FinLabel algorithm and exploiting the subregularities in the associated generating trees.  We consider some general kinds of generating trees for which the FinLabel algorithm fails to determine in a finite number of iterations the generating function that enumerates the underlying class of permutations.  Our procedure provides a unified approach in these cases leading to a system of equations satisfied by a certain finite set of generating functions which can be readily solved with the aid of programming.
	\end{abstract}
	
	\date{\today}
	
	\maketitle

	\section{Introduction}

	The study of pattern avoidance in permutations has been an object of ongoing interest to combinatorists over the past few decades. See, for example, the text \cite{kitaev1} for a general review of main results, techniques and directions. We seek a general procedure for enumerating broad classes of pattern restricted permutations.  Previously, an automatic approach to counting members of an avoidance class via enumeration schemes was initiated by Zeilberger \cite{Zeil} and later applied to a variety of problems (see \cite{Bax1,Bax2,BP,Pud} and references contained therein).   In \cite{BNZ}, further algorithms were found that derived functional equations for the generating functions automatically which enabled polynomial-time enumeration for a set of consecutive patterns. In \cite{NZ}, the more general problem of counting permutations according to the number occurrences of a pattern was undertaken using an automatic approach and the problem for patterns of length three was considered in detail.

Here, we revisit the classical avoidance problem for various sets of patterns and provide a somewhat general procedure through an in-depth analysis of certain kinds of generating trees. Recall that each node of a generating tree corresponds to a combinatorial object, and the branch leading to a node encodes a particular choice made in the construction of the object. Certain families of combinatorial objects admit a recursive description in terms of generating trees \cite{bar0, bar1, bar2, CGHK}, which frequently leads to the enumeration of the object in question, related explicit formulas and efficient random generation algorithms \cite{M2}.

Generating trees were first utilized in the enumeration of subclasses of permutations by West \cite{W1, W2} in the context of pattern avoidance and have been further exploited in closely related problems \cite{bar3}.
Later in \cite{V2,V1}, the generating tree idea was developed in the context of restricted permutations and powerful algorithms were found that can automatically produce the related generating functions, in particular, in the case when the associated generating trees are of finite size. Herein, we demonstrate how to extend these algorithms applicable only to finitely labeled generating trees to study several families in which they are infinitely labeled.  This is achieved through a more in-depth understanding of a form of {\em subregularity} in the tree structure.

	We put forth in this section a preliminary discussion following \cite{V2} and demonstrate how the enumeration of permutations avoiding a given set of patterns can be described in terms of counting paths within the corresponding generating tree. A few definitions are in order. Recall that a {\em generating tree} is a rooted, labeled tree such that the label of a node determines the labels of its children, if any. The nature of the labels of the nodes is immaterial and, as can be seen in our context, we use permutations to label the nodes. To specify a generating tree, it suffices to identify: (i) the label of the root, and (ii) a set of succession rules of the form
	\beq
	(l)\rightsquigarrow(l_1),(l_2),\ldots,(l_s)
	\feq
	describing how to label the nodes connected by the edges emanating from a node with label $(l)$ using the labels $(l_1),(l_2),\ldots,(l_s)$. One often refers to the label on the left side of the rule as the \emph{parent} and the labels on the right as the \emph{children}, with these terms applying to the nodes themselves as well. For example, each node in an infinite complete binary tree has two children, hence it is enough to use only one label, say $(2)$, leading to the following description:
	\begin{itemize}
		\item[]  {Root}: $(2)$
		\item[]  {Rule}: $(2)\rightarrow(2)(2)$.
	\end{itemize}
	\par
	We use the standard notation $\nn$ and $\cc$ to refer to the sets of natural and complex numbers, respectively. Let $[n]=\{1, \ldots, n\}$ for $n\in \nn$, with $[0]=\varnothing$. Also, for any word $\sigma$ of length $n\in \nn$, let $\sigma(i)$ represent its $i$-th entry for $i\in [n]$. Similarly, for any matrix $\M$, $\M(i,j)$ will denote its $(i,j)$-th entry.
	A \emph{permutation} of the set $[n]$ is any arrangement of the elements of $[n]$. We denote the set of all permutations of $[n]$ by $S_n$ and let $S:=\cup_{j\geq 1} S_j$ be the set of all permutations of finite length. Throughout this paper, for any $\pi\in S_n$, $|\pi|:=n$ refers to the length of the permutation $\pi$. For $\tau=\tau(1)\tau(2)\cdots\tau(k)\in S_k$ and $\sigma=\sigma(1)\sigma(2)\cdots\sigma(n)\in S_n$, we say that the permutation $\sigma$ \emph{contains} $\tau$ as a \emph{pattern} if there exist indices $1\leq i_1 <i_2<\cdots<i_k\leq n$ such that $\sigma(i_a)<\sigma(i_b)$ if and only if $\tau(a)<\tau(b)$ for all $1\leq a,b\leq k$. Otherwise, it is said that $\sigma$ \emph{avoids} $\tau$. We denote the set of all permutations in $S_n$ that avoid the pattern $\tau$ by $S_n(\tau)$, and similarly, define $S(\tau):=\cup_{j\geq1} S_j(\tau)$ as the set of all permutations avoiding $\tau$. More generally, for a set $1\notin B \subset S$ of patterns, we use the notation $S_n(B):=\cap_{\tau\in B}S_n(\tau)$ and $S(B):=\cap_{\tau\in B}S(\tau)$ to refer to the set of permutations of a given length or of any length, respectively, avoiding all patterns in the set $B$. Our interest here is to find the number of permutations in  $S_n(B)$, i.e., $|S_n(B)|$, or equivalently to study the corresponding generating function
	\beq
	G_B(x) := \sum_{n\geq1}|S_n(B)|x^n, \qquad   x\in \cc.
	\feq

To establish a useful connection between generating trees and the avoidance problem in permutations, we define a pattern-avoidance tree $\T(B)$ for a given set of patterns $B$ as follows. The tree $\T(B)$ is understood to be empty if there is no permutation of arbitrary length avoiding the set $B$.  Otherwise, $1\notin B$ and the root can always be taken as $1$, i.e., $1\in \T(B)$. Starting with this root, the remainder of the  tree $\T(B)$ can then be constructed in a recursive manner. To this end, we let the $n$-th level of the tree consist precisely of the elements of $S_n(B)$ arranged in such a way that the parent of a permutation $\pi:=\pi(1)\cdots\pi(n) \in S_n(B)$ for which $\pi(j)=n$ for some $1 \leq j \leq n$ is the unique permutation $\pi':=\pi'(1)\cdots\pi'(n-1)\in S_{n-1}(B)$ where $\pi'(i)=\pi(i)$ for $i \in [j-1]$ and $\pi'(i)=\pi(i+1)$ for $i \in [j,n-1]$. See Figure \ref{figT123a} for the first few levels of $\T(\{123\})$. A simple but important observation is that the size of $S_n(B)$ is equal to the number of nodes in the $n$-th level of $\T(B).$
	\begin{figure}[htp]
		{\footnotesize
			\begin{forest}
				for tree={fit=band,}
				[1[21,[321,[4321] [3421] [3241] [3214]] [231,[4231] [2431]] [213,[4213], [2413],[2143]]] [12,[312,[4312] [3412] [3142]] [132,[4132] [1432]]]]
		\end{forest}}
		\caption{First four levels of $\T(\{123\})$}\label{figT123a}
	\end{figure}
	
	Hence, we focus on an understanding of the nature of this tree, more specifically, {\em subregular structures} contained within it. More precisely, let $\T(B;\pi)$ denote the subtree consisting of $\pi$ and its descendants in $\T(B)$. For any $1 \leq m <n\in \nn$, we say that the node (labeled by) $\pi\in S_n(B)$ is {\em reducible} to the node $\pi'\in S_m(B)$ if the subtrees starting from $\pi$ and $\pi'$ are isomorphic, i.e., $\T(B;\pi)\cong\T(B;\pi')$. For instance, it is seen that $\T(\{123\};12)\cong\T(\{123\};1)$ and $\T(\{123\};312)\cong\T(\{123\};21)$. Suppose $t$ is the length of the longest pattern in $B$. Then, from \cite{V1}, we have $\T(B;\pi)\cong\T(B;\pi')$ for $\pi, \pi' \in S(B)$ if and only if, for each $1\leq j\leq t$, the number of nodes in the $j$-th level of subtree $\T(B;\pi)$ is equal to the number of nodes in the $j$-th level of subtree $\T(B;\pi')$.
	
	Now, based on this subregularity concept, we form the tree denoted by $\T[B]$ which is an isomorphic copy of $\T(B)$ wherein the nodes belonging to the same irreducible class are labeled the same. Clearly, $\T[B]$ is a generating tree whose labels correspond exactly to the isomorphism classes of $\L[B]:=\{\T(B; \pi)|\pi \in S(B)\}$. We let $\R[B]$ denote the set of succession rules for this generating tree. For instance, the first few levels of $\T[\{123\}]$ are given in Figure \ref{figT123b}.
		\begin{figure}[htp]
		{\footnotesize
			\begin{forest}
				for tree={fit=band,}
				[1[21,[321,[4321] [1] [21] [321]] [1,[21] [1]] [21,[321] [1 21]]] [1,[21,[321] [1] [21]] [1,[21] [1]]]]
		\end{forest}}
		\caption{First four levels of $\T[\{123\}]$}\label{figT123b}
	\end{figure}
	
	For any generating tree $\T[B]$, we define the directed graph $\D[B]$ whose vertices correspond to the set of all isomorphism classes of labels in $\T(B)$. An edge from the label $\alpha$ to the label $\beta$ exists if and only if the rule $\alpha\rightsquigarrow \beta$ belongs to the set of succession rules $\R[B]$.  For instance, the graph $\D[\{123\}]$ is depicted in Figure \ref{figT123c}. Note that multiple edges occurring between $\alpha$ and $\beta$ corresponds to the case when $\beta$ arises more than once as a child of $\alpha$.
	\begin{figure}[htp]
		\begin{picture}(0,40)(70,-10)
			\put(0,0){$1$}\put(30,0){$21$}\put(65,0){$321$}
			\put(110,0){$4321$}\put(7,3){\vector(1,0){20}}\put(27,6){\vector(-1,0){20}}
			\put(41,3){\vector(1,0){20}}\put(61,6){\vector(-1,0){20}}
			\put(84,3){\vector(1,0){20}}\put(104,6){\vector(-1,0){20}}\put(135,3){$\ldots$}
			\put(0,2){\qbezier(6,7)(3,28)(-2,7)\put(-1,7){\vector(0,-1){1}}}
			\put(32,2){\qbezier(6,7)(3,28)(-2,7)\put(-1,7){\vector(0,-1){1}}}
			\put(70,2){\qbezier(6,7)(3,28)(-2,7)\put(-1,7){\vector(0,-1){1}}}
			\put(120,2){\qbezier(6,7)(3,28)(-2,7)\put(-1,7){\vector(0,-1){1}}}
			\qbezier(70,-2)(37,-6)(4,-2)\put(4,-2){\vector(-1,0){1}}
			\qbezier(120,-2)(62,-11)(4,-5)\put(4,-5){\vector(-1,0){1}}

\qbezier(120,-2)(76,-8)(34,-2)\put(34,-2){\vector(-1,0){1}}
		\end{picture}
		\caption{Directed graph $\D[\{123\}]$}\label{figT123c}
	\end{figure}
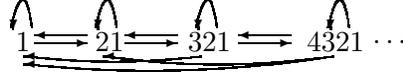	

	We equip the set $\L[B]$ of all isomorphism classes of labels with the lexicographical ordering wherein each permutation of length $k$ appears before all permutations of length $k+1$ for all $k$. For example, the nodes of $\D[\{123\}]$ are ordered as $1,21,321,\ldots$\,. We then define $\M[B]$ for the graph $\D[B]$ as the matrix whose entries are given by $\M[B](v,w)=s$ for all $v,w \in \D[B]$, where $s$ is the number of edges from $v$ to $w$. $\M[B]$ is referred to as the \emph{transition matrix} of the graph $\D[B]$ and clearly has non-negative integral entries.  For instance, $\M[\{123\}]$ is given by the following infinite matrix
	\beq
	\M[\{123\}] = \left(\begin{array}{llllll}
		1&1&0&0&0&\cdots\\
		1&1&1&0&0&\cdots\\
		1&1&1&1&0&\cdots\\
		1&1&1&1&1&\cdots\\
		\vdots&\vdots&\vdots&\vdots&\ddots
	\end{array}\right).
	\feq
	
	It is seen that the number of permutations in $S_n(B)$ is equal to the number of paths of length $n-1$ starting at the node $1$ in the graph $\D[B]$. Hence, the transfer-matrix method \cite[Theorem 4.7.2]{Stanley1} implies the generating function $G_B(x)$ is given by
	\beqn \label{general_tm}
	G_B(x)=(1,0,0,\ldots)\sum_{n\geq1}(\M[B])^{n-1}x^n(1,1,1,\ldots)^t,
	\feqn
	where $v^t$ denotes the transpose of the vector $v$. If the set of all isomorphism classes $\{\T(B; \pi) | \pi \in S(B)\}$ is finite (i.e., if $\M[B]$ is finite), then \eqref{general_tm} implies the generating function $G_B(x)$ is rational and equal to
	\beqn \label{finite_tm}
	G_B(x)=x(1,0,0,\ldots)(I-x\M[B])^{-1}(1,1,1,\ldots)^t.
	\feqn	
In \cite{Kremer1} and \cite{V1}, it was shown that the set of isomorphism classes is finite if and only if $B$ contains both a child of an increasing permutation and a child of a decreasing permutation. Furthermore, a Maple package has been developed (described in \cite{V1} and available at http://math.rutgers.edu/$\sim$vatter/) that finds the generating function in this case. We will refer to this package (algorithm) as the FinLabel algorithm throughout this paper.
	
	We close this section with a simple finite case example.
	
	\begin{example}\label{exab1}
		Let $B=\{123,43215\}$. Then the rules of $\T[B]$ are given by
		\beq
		\R[B] &=& \{ 1\rightsquigarrow1,21 \} \\
		&&  \cup \{ 21\rightsquigarrow1,21,321 \} \\
		&& \cup \{ 321\rightsquigarrow1,21,321,321 \},
		\feq
		with the root $1$. Thus $\D[B]$ is as presented in Figure \ref{figTexa}.
		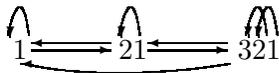
\begin{figure}[htp]
			\begin{picture}(0,50)(70,-10)
				\put(0,0){$1$}\put(40,0){$21$}\put(85,0){$321$}
				\put(7,3){\vector(1,0){30}}\put(37,6){\vector(-1,0){30}}
				\put(51,3){\vector(1,0){30}}\put(81,6){\vector(-1,0){30}}
				\put(0,2){\qbezier(6,7)(3,28)(-2,7)\put(-1,7){\vector(0,-1){1}}}
				\put(42,2){\qbezier(6,7)(3,28)(-2,7)\put(-1,7){\vector(0,-1){1}}}
				\put(90,2){\qbezier(6,7)(3,28)(-2,7)\put(-1,7){\vector(0,-1){1}}}
				\put(94,2){\qbezier(6,7)(3,28)(-2,7)\put(-1,7){\vector(0,-1){1}}}
				\qbezier(82,-2)(37,-9)(4,-2)\put(4,-2){\vector(-1,0){1}}
			\end{picture}
			\caption{Directed graph $\D[\{123,43215\}]$}\label{figTexa}
		\end{figure}

		The matrix $\M[B]$ is given by
		\beq
		\M[B]=\left(\begin{array}{lll}
			1&1&0\\
			1&1&1\\
			1&1&2
		\end{array}\right).
		\feq
		Hence, by \eqref{finite_tm}, the generating function $G_B(x)$ is equal to
		\beq
		x(1,0,0)(I-x\M[B])^{-1}(1,1,1)^t=\frac{x(1-2x)}{(1-x)(1-3x)},
		\feq
		as expected (see Theorem 3.1 in \cite{CW}).
	\end{example}
	
	However, when the matrix $\M[B]$ is infinite, the evaluation of $G_B(x)$ can be an intricate task which we will focus on in this paper.  In the next section, we develop an algorithm for computing $\T[B]$, $\D[B]$ and $\M[B]$ after finitely many iterations which is applicable to cases when $\M[B]$ is infinite.  In the third section, we apply this algorithm together with a simple general enumerative result to deduce $G_B(x)$ for several classes of pattern sets $B$ for which $\D[B]$ belongs to one of three general families of graphs.
	
	\section{Infinite size $\M[B]$: subregularity structures}
	
	To study $\M[B]$ of infinite size, we define $P_{n}(B; \pi)$ to be the number of nodes at the $n$-th level of $\T(B;\pi)$. Let $F_\pi(x)$ be given by
	\beqn \label{ss}
	F_\pi(x) := x^{|\pi|-1} \sum_{n=1}^\infty P_{n}(B; \pi) x^n.
	\feqn
	In words, $F_\pi(x)$ is the generating function that enumerates the paths beginning with the root $\pi$ of $\T(B;\pi)$. Clearly, $G_B(x) = F_1(x)$.  Note that $F_\pi(x)$ is analytic for all $\pi$ in some interval containing zero dependent upon $B$.  For any rule $v\rightsquigarrow w_1w_2\cdots w_s$, we have
	\beqn \label{PnBv}
	P_n(B;v) = \begin{cases}
		1, &\quad\text{if}\quad n=1;\\
		\sum_{j=1}^s P_{n-1}(B;w_j), &\quad\text{if}\quad  n\geq 2.
	\end{cases}
	\feqn
	Therefore, substituting \eqref{PnBv} into \eqref{ss} implies
	\beqn \label{F_rule}
	F_v(x)&=& x^{|v|} + x^{|v|-1} \sum_{n=2}^\infty \left( \sum_{j=1}^s P_{n-1}(B;w_j) \right)x^n \notag \\
	&=& x^{|v|} + \sum_{j=1}^s x^{|v|+1-|w_j|}F_{w_j}(x),
	\feqn
	where we exchange sums and apply definition \eqref{ss} to $w_j$ in obtaining the second equality.

	Our general approach will be to rewrite the set of succession rules as a set of equations of the form \eqref{F_rule} and then obtain information on $F_1(x)$.  We begin with a simple example to illustrate the general idea.
	
	\begin{example}\label{123,132}
		Let $B=\{123,132\}$. Here, the generating tree $\T[B]$ has the rules
		\beqn
		\R[B]&:=& \{1 \rightsquigarrow 21,12\} \label{ex2_rule1}\\
		&& \cup \{k(k-1)\cdots21\rightsquigarrow(12)^k,(k+1)k\cdots21\mid  k \geq 2\} \label{ex2_rule2} \\
		&& \cup \{k(k-1)\cdots312\rightsquigarrow(12)^{k-2},(k+1)k\cdots312\mid \  k \geq 2 \} \label{ex2_rule3},
		\feqn
		with the ordered set
		$$
		\mathcal{L}[B]=\{1,21,12,321,312,4321,4312,\ldots\},
		$$
		and the infinite matrix $\M[B]$ given by
		$$\M[B]=\left(\begin{array}{llllllllll}
			0&1&1&0&0&0&0&0&0&\cdots\\
			0&0&2&1&0&0&0&0&0&\\
			0&0&0&0&1&0&0&0&0&\\
			0&0&3&0&0&1&0&0&0&\cdots\\
			0&0&1&0&0&0&1&0&0&\\
			0&0&4&0&0&0&0&1&0&\cdots\\
			0&0&2&0&0&0&0&0&1&\\
			\vdots&&&\vdots&&&&\vdots
		\end{array}\right).$$

		Since $\M[B]$ is not finite, the FinLabel algorithm fails to count the members of $S_n(B)$.		
		However, we can describe $\M[B]$ as an infinite system of equations as follows. Observe that, by \eqref{F_rule}, the rule \eqref{ex2_rule1} can be written as $F_1(x)=x+F_{21}(x)+F_{12}(x)$. Similarly, we can write rules \eqref{ex2_rule2} and \eqref{ex2_rule3} as
		\beqn
		&&F_{k(k-1)\cdots1}(x)=x^k+kx^{k-1}F_{12}(x)+F_{(k+1)k\cdots1}(x),\qquad k\geq2, \label{exm2-1}
		\feqn
		and
		\beqn
		&&F_{k(k-1)\cdots312}(x)=x^k+(k-2)x^{k-1}F_{12}(x)+F_{(k+1)k \cdots 312}(x), \qquad k\geq2. \label{exm2-2}
		\feqn		
		Consider the set of equations \eqref{exm2-1} for all $k\geq 2$. By addition of the left and right sides of all the equations, and cancellation of like terms on both sides, we obtain
		\beqn \label{exm2-eq1}
		&&  F_{21}(x)=\frac{x^2}{1-x}+F_{12}(x)\sum_{k\geq2}kx^{k-1}.
 \feqn
		In a similar manner, \eqref{exm2-2} implies
		\beqn \label{exm2-eq2}
		&& F_{12}(x)=\frac{x^2}{1-x}+F_{12}(x)\sum_{k\geq2}(k-2)x^{k-1}.
		\feqn
		Solving the system \eqref{exm2-eq1} and \eqref{exm2-eq2} for $F_{12}(x)$ and $F_{21}(x)$ yields
		\beq
		F_{12}(x)=\frac{x^2(1-x)}{1-2x}\mbox{ and }F_{21}(x)=\frac{x^2(1+x)}{1-2x}.
		\feq
		Hence, we have
		\beq
		G_B(x)=F_1(x)= x+F_{12}(x)+F_{21}(x) = \frac{x}{1-2x},
		\feq
	which is in accordance with \cite{SS}.
\end{example}

	Since the determination of the rule set $\R[B]$ is the key to our process, we will formulate a simple algorithm which draws upon and extends FinLabel as follows. To this end, we say that a set $R$ of $m$ rules constructs a general rule with index $k$ if the general rule with $k=1,2,\ldots,m$ gives exactly all the rules in the set $R$. For example, the set $$R=\{1\rightsquigarrow1,12;\, 12\rightsquigarrow1^2,123;\, 123\rightsquigarrow1^3,1234;\,\ldots; 12\cdots 50\rightsquigarrow1^{50},12\cdots51\}$$ of $50$ rules constructs the general rule $12\cdots k\rightsquigarrow 1^k,12\cdots(k+1)$. The question is how many rules should be included in $R$ so that the correct corresponding general rule can be extracted. The next observation implies that we only need a finite number of rules to construct a general rule.
	
	\begin{proposition}\label{prov2}
		Let $B\subset S$ be any finite set of patterns and let $t$ be the length of its longest pattern. Suppose $R$ is a set of $m$ rules that is a subset of the rules of the tree $\T[B]$ which constructs a general rule parameterized by index $k$. If $m\geq t-1$, then the general rule holds in $\T[B]$ for all $k$.
	\end{proposition}
	\begin{proof}
		This is equivalent to the statement that for any two given permutations $\pi, \pi' \in S(B)$, we have $\T(B;\pi)\cong\T(B;\pi')$ if and only if for each $1\leq j\leq t$, the number of nodes at the $j$-th level of subtree $\T(B;\pi)$ is equal to the number of nodes of the $j$-th level of subtree $\T(B;\pi')$.
	\end{proof}
	
	Using this proposition, we can guarantee that the following procedure calculating $\T[B]$, $\D[B]$ and $\M[B]$ exits after finitely many iterations. The procedure is comprised of two main parts. For the first part, we find the set of rules of the generating tree $\T[B]$. To that end, Step (II) proposes to use a small modification of the algorithm FinLabel from \cite{V1}, where the algorithm stops after $D$ units of time and outputs the set $\mathcal{R}_D$. Next, step (III) searches for a set of rules constructing a general one by looking at rules with an offspring of form either $\alpha m(m+1)\cdots m'\beta\in S_{m'}$ or $\alpha m'(m'-1)\cdots m\beta\in S_{m'}$ for some $m'>m$. Then Step (IV) verifies whether or not we obtained exactly the set $\mathcal{R}[B]$. If not, then the algorithm increases the value of $D$ and reruns Steps (II) and (III). If yes, then we have successfully found the generating tree $\T[B]$. As seen below, $D$ depends upon the pattern set $B$ and theoretically might need to be quite large, which would preclude the possibility of simply guessing the set of rules. After calculating $\T[B]$, the algorithm proceeds with the second part in which the generating function $G_B(x)$ is computed.
	
	\begin{algorithm}
		\caption{Calculating $\R[B]$ and $\D[B]$}\label{alg1}
		\begin{algorithmic} [hbt!]
			\State {\bf (I): Input}  Let $1\notin B\subset S$ be any set of patterns and $D\geq2$.
			\State {\bf (II)}: Let $P=\{1\}$ and $\mathcal{R}_D=\emptyset$ as specified by the FinLabel algorithm
 given in \cite{V1}. Then run the FinLabel algorithm $D$ iterations to update $P$ and $\mathcal{R}_D$.		
			\State {\bf (III)}:  Let  $\mathcal{R}':=\mathcal{R}_D[B]$. Construct one or more general rules from
 $R\subseteq\mathcal{R}'$ through use of Proposition \ref{prov2}; then remove $R$ from $\mathcal{R}'$ and replace with general rule(s).
\State {\bf (IV)}: Using Proposition \ref{prov2} and induction, we attempt to show
$\mathcal{R}'=\mathcal{R}[B]$. If successful, proceed to {\bf (V)}.   Otherwise, increase $D$ by one and return to {\bf (II)}.
			\State {\bf (V)}: Find $\D[B]$ and the associated matrix $\M[B]$ for $\T[B]$.
		\end{algorithmic}
	\end{algorithm}
	
	We illustrate how the algorithm works with the following two examples.
	
	\begin{example}
		Let $B=\{123,132\}$ and $D=200$. Note that step (II) gives the set of rules $\mathcal{R}_D=\{1\rightsquigarrow21,12;\,12\rightsquigarrow312;\,21\rightsquigarrow12^2,321;\,
		312\rightsquigarrow12,4312;\,321\rightsquigarrow12^3,4321;\,4312\rightsquigarrow12^2,54312;
		\,4321\rightsquigarrow12^4,54321;\ldots\}$. Then step (III) outputs \begin{align*}
			\mathcal{R}[B]&=\{k(k-1)\cdots312\rightsquigarrow12^{k-2},(k+1)k\cdots312\mid k\geq2\}\\
			&\cup\{k(k-1)\cdots21\rightsquigarrow12^k,(k+1)k\cdots21\mid k\geq1\}.
		\end{align*}
		By an induction argument with respect to $k$, we have that the generating tree $\T[B]$ is indeed given by the rules $\mathcal{R}[B]$. The associated matrix $\M[B]$ is given in Example \ref{123,132}.
	\end{example}
	
	\begin{example}
		Let $B=\{123,1432,2143\}$. Using the algorithm above with $D=200$, we obtain
		\begin{align*}
			\mathcal{R}[B]&=\{k(k-1)\cdots4132\rightsquigarrow12^{k-3},(k+1)k\cdots4132\mid k\geq3\}\\
			&\cup\{k(k-1)\cdots21\rightsquigarrow12^k,(k+1)k\cdots21\mid k\geq1\} \\
			&\cup\{12\rightsquigarrow21,132\} \\
			& \cup \{ 1\rightsquigarrow 12,21\}.
		\end{align*}
		We rewrite the rule $k(k-1)\cdots4132\rightsquigarrow12^{k-3},(k+1)k\cdots4132$ as
		\beq
		&&F_{k(k-1)\cdots4132}(x)=x^k+(k-3)x^{k-1}F_{12}(x)+F_{(k+1)k\cdots4132}(x),\,\qquad k\geq3. \label{exm5-1}
		\feq
		Thus, proceeding similarly as in the derivation of \eqref{exm2-eq1}, we obtain
		\beqn \label{exm5-eq1}
		F_{132}(x)=\frac{x^3}{1-x}+F_{12}(x) \sum_{k\geq3}(k-3)x^{k-1}.
		\feqn
		Likewise, the rule $k(k-1)\cdots21\rightsquigarrow12^k,(k+1)k\cdots21$ yields
		\beq
		&&F_{k(k-1)\cdots1}(x)=x^k+kx^{k-1}F_{12}(x)+F_{(k+1)k\cdots1}(x),\,\qquad k\geq1, \label{exm5-1}
		\feq
		and hence
		\beqn \label{exm5-eq2}
		F_{1}(x)&=\frac{x}{1-x}+F_{12}(x)\sum_{k\geq1}kx^{k-1}.
		\feqn
		By the last rule above and upon taking $k=1$ in the formula for $F_{k(k-1)\cdots1}(x)$, we have
		\beqn
		&&F_{12}(x)=x^2+xF_{21}+F_{132}(x), \label{exm5-eq3} \\
		&&F_1(x)=x+F_{12}(x)+F_{21}(x). \label{exm5-eq4}
		\feqn
		Thus, solving the system \eqref{exm5-eq1}-\eqref{exm5-eq4} for $F_1(x)$ gives
		\beq
		G_B(x) =F_1(x)=\frac{x}{1-2x-x^2}.
		\feq
	\end{example}
	
	Let us say that $a(x)$ is a rational linear combination of $b_1(x),\ldots,b_s(x)$ if there exist rational functions $c_j(x)$ such that $a(x)=c_0(x)+\sum_{j=1}^sc_j(x)b_j(x)$.  In order to systematically leverage Algorithm \ref{alg1}, we will need the following simple yet important result.
	
	\begin{theorem}\label{thg2}
		Let $1\notin B$ be any set of patterns and $m\geq1$ be a natural number. Suppose that for any node $\pi$ of $\D[B]$ with $|\pi|=:m$, the generating function $F_\pi(x)$ can be expressed as a rational linear combination of the $F_{\pi'}(x)$ with $\pi'\in \D[B]$ and $|\pi'|\leq m$ such that $c_j(0)=0$ for all $j$ in the corresponding coefficients $c_j(x)$. Then $G_B(x)$ is a rational generating function.
	\end{theorem}
	\begin{proof}
		Since $1\notin B,$ there is a rule in $\R[B]$ with parent $1$. Further, the result is apparent if $m=1$, so we may assume $m \geq 2$. Define $R'$ to be the set of all rules in $\R[B]$ whose parents are of length  at most $m-1$. For any rule $v\rightsquigarrow w_1w_2\cdots w_s$ in $R'$, we rewrite it in the form of \eqref{F_rule}:
		\beqn \label{Fv_eq}
		F_v(x)=x^{|v|}+\sum_{j=1}^s x^{|v|+1-|w_j|}F_{w_j}(x),
		\feqn
	where $|w_j|\leq m$ for any $1\leq j \leq s$. By hypothesis, for each $\pi\in \D[B]$ with $|\pi|=m$, $F_\pi(x)$ is a rational linear combination of $F_v(x)$ with $|v|\leq m$.  Therefore, combining these equations with those given in \eqref{Fv_eq} for $|v|\leq m-1$, one obtains a linear system of equations in the variables $F_v(x)$ where $|v| \leq m$. The associated coefficient matrix of this system has rational function coefficients and non-zero determinant since each main diagonal entry is of the form $1-xf(x)$ for some rational $f(x)$, with each entry below the diagonal seen to be a multiple of $x$ (possibly zero).  Indeed, the determinant is equal to $1$ at $x=0$, and hence by continuity, there exists some open interval containing zero for which the determinant is non-zero. Thus, Cramer's rule implies that each component of the solution of the system (valid for all $x$ on the interval) is a rational function. In particular, $G_B(x)=F_1(x)$ is rational, as desired.
	\end{proof}
	
	\section{Enumeration results for families of patterns}
	
	In this section, we use Algorithm \ref{alg1} along with Theorem \ref{thg2} to study several families of sets of patterns whose corresponding graphs are infinite and hence FinLabel is not applicable directly.  In each of these cases, let $G=(V,E)$ be a directed graph with set of nodes $V\subset S$ and edge set $E\subset V\times V$.
	
	\subsection{Almost path-directed graphs}
	
	We say that $G$ is an {\em almost path-directed graph} if the set of nodes $V$ can be partitioned as $V:=V'\cup W$, where $V':=\{v_0,v_1,\ldots\}$ and $|v_i|=i+m$ for all $i \geq 0$ with $m \geq 1$ fixed such that
	\begin{itemize}
		\item  The length of any node $w$ in $W=V\backslash V'$ is at most $m$.
		\item $(v_j,v_{j+1})\in E$ for all $j\geq0$ and is not repeated.
		\item All other edges $(c,d) \in E$ may be repeated (with $c=d$ possible) and are such that $c$ or $d$ belongs to $W$ (possibly both).
	\end{itemize}

	In this context, $V'$ and $W$ will be referred to as the parameters of the almost path-directed graph $G$.  Note that $(e_1,e_2)\in E$ with $e_1 \in W$ implies $e_2 \in W\cup\{v_0,v_1\}$ since $|e_1|\leq m$. In practice, we frequently have for each $m' > m$ that the node $v_{m'-m}\in V'$ is obtained from $v_0$ by replacing $m$ with either $m(m+1)\cdots m'$ or $m'(m'-1)\cdots m$. Furthermore, the result below is seen to apply more generally to any generating function which enumerates paths starting from the root and having $n-1$ steps for $n \geq 1$ in an almost path-directed graph $G$ independent of whether or not $G$ arose in the context of pattern avoidance. A similar remark applies to the graphs discussed in the subsequent two sections.
	
	When $B$ is a set of patterns whose directed graph $\D[B]$ is almost path-directed, we may apply Algorithm \ref{alg1} together with the following result to ascertain the generating function.
	
	\begin{theorem}\label{thres1}
		Let $1\notin B\subset S$ be any set of patterns. Supposed $\D[B]$ is an almost path-directed graph with parameters $V'$ and $W$. If for each $w\in W$, the generating function $\sum_{j\geq0}\M[B](v_j,w)x^j$ is rational, then $G_B(x)$ is rational.
	\end{theorem}
	\begin{proof}	
		We first show that $F_{v_0}(x)$ is a rational linear combination of the $F_w(x)$ with $w\in W$ such that $c_j(0)=0$ for each corresponding coefficient $c_j(x)$. To this goal, since $\D[B]$ is almost path-directed, we have
		\beq
		F_{v_j}(x)=x^{|v_0|+j}+\sum_{w\in W}\M[B](v_j,w)x^{|v_0|+j+1-|w|}F_w(x)+F_{v_{j+1}}(x), \qquad j \geq 0.
		\feq
		Hence, by summing over $j\geq0$ and using the fact that $F_{v_j}(x)\rightarrow0$ as $j\rightarrow\infty$ for $x$ sufficiently close to zero, we obtain
		\beqn \label{Fv0_e}
		F_{v_0}(x)=\frac{x^{|v_0|}}{1-x}+\sum_{w\in W}x^{|v_0|+1-w}\left(\sum_{j\geq0}\M[B](v_j,w)x^{j}\right)F_w(x).
		\feqn
	
	By assumption, each of the functions $\sum_{j\geq0}\M[B](v_j,w)x^j$ is rational, and hence by \eqref{Fv0_e}, $F_{v_0}(x)$ is a rational linear combination of the $F_w(x)$ with $w\in W$ of the desired form and $W$ a finite set.  By a similar argument (starting all sums from $j=1$), the same holds for $F_{v_1}(x)$.  Note that the set of all nodes of length at most $m$ is given by $W'=W\cup\{v_0\}$ since $\D[B]$ is almost path-directed. Then $W'$ is seen to meet the conditions of Theorem \ref{thg2} concerning nodes of length $m$, which implies the stated result.
	\end{proof}
		
	\begin{example}
		Let $B=\{123,312\}$.  Then Algorithm \ref{alg1} outputs
		\beq
		\R[B]&=&\{12\rightsquigarrow12\} \\
		&& \cup\{k(k-1)\cdots1\rightsquigarrow12^k,(k+1)k\cdots1\mid k\geq1\}.
		\feq
		See Figure \ref{figgt1} for the schematic of the corresponding directed graph which is almost path-directed. One may verify that the parameters in this case are $V'=\{v_j:=(j+2)(j+1)\cdots1\mid j\geq0 \}$ and $W=\{1,12\}$.
		We also have that
		$\sum_{j\geq0}\M[B]((j+2)(j+1)\cdots1,12)x^j$ is the rational function $\sum_{j\geq0}(j+2)x^j=\frac{2-x}{(1-x)^2}$ (with $\sum_{j\geq0}\M[B]((j+2)(j+1)\cdots1,1)x^j=0$). Therefore, Theorem \ref{thres1} implies that $G_B(x)$ is a rational generating function. Moreover, it allows us to calculate $G_B(x)$ by solving the following system:
		\beq
		&& F_1(x)=x+F_{12}(x)+F_{21}(x),\\
		&& F_{12}(x)=x^2+xF_{12}(x),\\
		&& F_{21}(x)=\frac{x^2}{1-x}+\frac{x(2-x)}{(1-x)^2}F_{12}(x).
		\feq
		Hence,
		\beq
		G_B(x)=F_1(x)=\frac{x}{1-x}+\frac{x^2}{(1-x)^3}.
		\feq
	\end{example}
	\par
	
	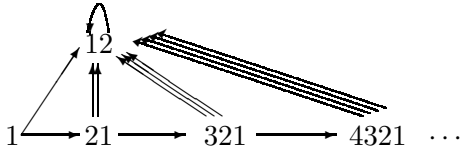
\begin{figure}[htp]
		\begin{picture}(120,55)
			\put(0,10){$1$}\put(30,10){$21$}\put(30,45){$12$}
			\put(75,10){$321$}\put(130,10){$4321$}\put(160,10){$\cdots$}
			\put(6,14){\vector(1,0){22}}
			\put(6,14){\vector(2,3){22}}
			\put(95,14){\vector(1,0){30}}
			\put(43,14){\vector(1,0){25}}
			\put(33,46){\qbezier(6,7)(3,28)(-2,7)\put(-1,7){\vector(0,-1){1}}}
			\put(33,21){\vector(0,1){20}}\put(35,21){\vector(0,1){20}}
			\put(76,21){\vector(-3,2){34}}\put(79,21){\vector(-3,2){35}}\put(82,21){\vector(-3,2){36}}
			\put(0,0){\qbezier(135,21)(93,35)(50,49)\put(50,49){\vector(1,1){2}}}
			\put(3,1){\qbezier(135,21)(93,35)(50,49)\put(50,49){\vector(1,1){2}}}
			\put(6,2){\qbezier(135,21)(93,35)(50,49)\put(50,49){\vector(1,1){2}}}
			\put(9,3){\qbezier(135,21)(93,35)(50,49)\put(50,49){\vector(1,1){2}}}
		\end{picture}
		\caption{Direct graph $\D[\{123,312\}]$}\label{figgt1}
	\end{figure}

	\begin{example}
		Let $B=\{123,2143\}$. Applying Algorithm \ref{alg1} yields $$\mathcal{R}[B]=\{k(k-1)\cdots1\rightsquigarrow1^k,(k+1)k\cdots1\mid k\geq1\}.$$
		Note that $\D[B]$ is an almost path-directed graph with parameters $V'=\{v_j=(j+2)(j+1)\cdots1\mid j\geq0\}$ and $W=\{1\}$. Thus, Theorem \ref{thres1} implies $G_B(x)$ is rational with associated linear system
		\beq
		&&F_1(x)=x+xF_1(x)+F_{21}(x),\\ &&F_{21}(x)=\frac{x^2}{1-x}+\frac{x^2(2-x)}{(1-x)^2}F_{1}(x).
		\feq
		Hence,
		\beq
		G_B(x)=F_1(x)=\frac{x-x^2}{1-3x+x^2}.
		\feq
	\end{example}
	
	\begin{example}\label{ex12331221543}
		Let $B=\{123,312,21543\}$. Then by Algorithm \ref{alg1}, we get
		\beq
		\R[B] &=& \{1\rightsquigarrow12,21 \} \\
		&& \cup \{ 12\rightsquigarrow12 \} \\
		&& \cup \{ 213\rightsquigarrow2143 \} \\
		&& \cup \{ j(j-1)\cdots1\rightsquigarrow12,213^{j-1},(j+1)j\cdots1\mid j\geq2 \}.
		\feq		
		Hence, $\D[B]$ is almost path-directed with parameters $V'=\{v_j=(j+3)(j+2)\cdots1\mid j\geq0\}$ and $W=\{1,12,21,213\}$. We then have the following system of equations:
		\beq
		&&F_1(x)=x+F_{12}(x)+F_{21}(x),\\
		&&F_{12}(x)=x^2+xF_{12}(x),\\
		&&F_{213}(x)=x^3+F_{2143}(x)=x^3+x^4,\\
		&&F_{321}(x)=\frac{x^3}{1-x}+\frac{x^2}{1-x}F_{12}(x)+F_{213}(x) \sum_{j\geq3}(j-1)x^{j-2}.
		\feq
		Hence, solving this system for $F_1(x)$ gives
		\beq
		G_B(x)=F_1(x)=\frac{x+x^3+x^4}{(1-x)^2}.
		\feq
	\end{example}
	
	In general, $\D[B]$ is almost path-directed for all pattern sets $B=\{123,312,21 k(k-1)\cdots3\}$ where $k\geq4$.
	
	\subsection{Backward path-directed graphs}
	We say $G$ is a {\em backward path-direct graph} if the set of nodes $V$ can be partitioned as $V:=V'\cup W$, where $V':=\{v_0,v_1,\ldots\}$ and $|v_i|=i+m$ for all $i \geq 0$ with $m \geq 1$ fixed such that
	\begin{itemize}		
		\item  The length of any node $w$ in $W= V\backslash V'$ is at most $m$.
		\item For each $j\geq0$, $(v_j,v_i)\in E$ for all $i=0,1,\ldots,j-1,j+1$, with all of these edges occurring once. Additionally, for some fixed integer $a\geq 0$, there are $a$ loops at the node $v_i$ for all $i \geq 0$.
		\item All other edges $(c,d) \in E$ may be repeated (with $c=d$ possible) and are such that $c$ or $d$ belongs to $W$ (possibly both).
	\end{itemize}
	
In this context, $V'$, $W$ and $a$ are called the parameters of the backward path-directed graph $G$.
In practice, we frequently have for each $m'>m$ that the node $v_{m'-m}\in V'$ is obtained from $v_0$ by replacing $m$ with either $m(m+1)\cdots m'$ or $m'(m'-1)\cdots m$.  A basic example of a backward path-directed graph is $D[\{123\}]$, where the tree $\T[\{123\}]$ is defined by the root $1$ and the set of succession rules $k(k-1)\cdots1\rightsquigarrow 1,21,\ldots,(k+1)k\cdots1$.

	\begin{example} \label{impex_11}
		Let $B=\{1243, 1324, 1342, 1423, 1432, 2143, 2413, 2431, 3142, 4132\}.$ Then the Algorithm \ref{alg1} describes the tree $\T[B]$ by the following succession rules:
		\beqn \label{exmp-rb-10}
		&& \R[B]=\{1\rightsquigarrow12,21\}\\
		&& \qquad \quad \cup \{12\rightsquigarrow21,123,132\} \notag\\
		&& \qquad \quad \cup \{ 21\rightsquigarrow21,213,321\}\notag\\
		&& \qquad \quad \cup \{ 123\rightsquigarrow21,123\}\notag\\
		&& \qquad \quad \cup \{ 213\rightsquigarrow21,123\}\notag\\
		&& \qquad \quad \cup \{ j(j-1)\cdots1\rightsquigarrow21,213,321,4321,\ldots,(j+1)j\cdots1\mid j\geq3\}.\notag
		\feqn
		One may verify that $\D[B]$ is backward path-directed with parameters $V'=\{v_j=(j+3)(j+2)\cdots1\mid j\geq 0\}$, $W=\{1,12,21,123,132,213\}$ and $a=1$.
	\end{example}
	
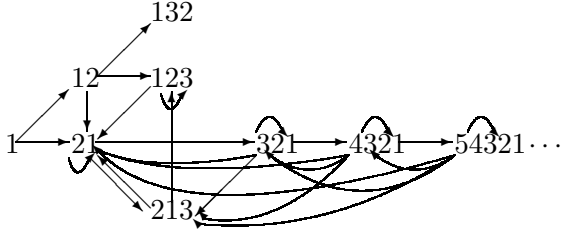
\begin{figure}[htp]
		\begin{picture}(180,100)
			\put(0,30){			\put(0,0){$1$}\put(25,0){$21$}\put(25,25){$12$}\put(95,0){$321$}\put(130,0){$4321$}
				\put(170,0){$54321$}\put(4,4){\vector(1,0){20}}
				\put(34,4){\vector(1,0){60}}\put(109,4){\vector(1,0){20}}
				\put(149,4){\vector(1,0){20}}\put(4,4){\vector(1,1){20}}
	\put(55,25){$123$}\put(55,50){$132$}\put(35,29){\vector(1,0){20}}
	\put(35,29){\vector(1,1){20}}\put(31,23){\vector(0,-1){15}}	
				\qbezier(24,-2)(27,-13)(32,-2)\put(32,-2){\vector(1,1){2}}
				\put(55,-25){$213$}\put(32,-2){\vector(1,-1){20}}
\put(35,24){\qbezier(24,-2)(27,-13)(32,-2)\put(32,-2){\vector(1,1){2}}}
\put(55,25){\vector(-1,-1){20}}\put(63,-19){\vector(0,1){42}}
				\put(55,-21){\vector(-1,1){20}} \qbezier(95,-1)(55,-8)(35,1)\put(35,1){\vector(-1,1){2}}\put(95,-1){\vector(-1,-1){23}}
			\qbezier(95,8)(100,19)(105,8)\put(105,8){\vector(1,-1){2}}
	        \qbezier(130,-1)(65,-12)(35,1)\qbezier(170,-1)(65,-32)(35,1)		\put(40,0){\qbezier(95,8)(100,19)(105,8)\put(105,8){\vector(1,-1){2}}}
\put(80,0){\qbezier(95,8)(100,19)(105,8)\put(105,8){\vector(1,-1){2}}}
				\put(198,0){$\cdots$}				\put(40,0){\qbezier(95,8)(100,19)(105,8)\put(105,8){\vector(1,-1){2}}}
\qbezier(130,-1)(115,-12)(100,-1)\put(100,-1){\vector(-1,1){2}} \put(40,0){\qbezier(130,-1)(115,-12)(100,-1)\put(100,-1){\vector(-1,1){2}}} \put(40,0){\qbezier(130,-1)(95,-28)(60,-1)\put(100,-1){\vector(-1,1){2}}}
\qbezier(130,-1)(103,-30)(75,-25)\put(75,-24){\vector(-1,1){2}}
\qbezier(170,-1)(123,-32)(73,-27)\put(73,-27){\vector(-1,1){2}}
}
		\end{picture}
		
		\caption{$D[\{1243, 1324, 1342, 1423, 1432, 2143, 2413, 2431, 3142, 4132\}]$}\label{exa1010}
		
	\end{figure}

	To determine $G_B(x)$ when $\D[B]$ is backward path-directed, we may apply Algorithm \ref{alg1} to determine $\T[B]$ and then employ the following result.
	
	\begin{theorem}\label{thres2}
		Let $1\notin B\subset S$ be any set of patterns. Suppose $\D[B]$ is a backward path-directed graph with parameters $V'$, $W$  and $a\geq0$. If $\sum_{j\geq0}\M[B](v_j,w)x^j$ is a rational generating function for any $w\in W$, then $G_B(x)$ is a rational generating function of $x$ and
		\beq
		t_0:=\frac{1+x-ax-\sqrt{(1+x-ax)(1-3x-ax)}}{2x(1+x-ax)}.
		\feq
	\end{theorem}
	\begin{proof}
		We first determine an expression for the generating function $F_{v_0}(x)$. Since $\D[B]$ is backward path-directed, formula \eqref{F_rule} yields
		\beqn \label{l1111}
		&& F_{v_j}(x)=x^{|v_0|+j}+axF_{v_j}(x)+F_{v_{j+1}}(x)+x^{j+1}F_{v_0}(x)+\cdots+x^2F_{v_{j-1}}(x) \notag \\
		&& \qquad \qquad + \sum_{w\in W}\M[B](v_j,w)x^{|v_0|+j+1-|w|}F_w(x), \qquad j \geq 0.
		\feqn
		
		Define $A(t):=\sum_{j\geq0}F_{v_j}(x)t^j$. Multiplying both sides of \eqref{l1111} by $t^j$, and summing over all $j\geq0$, we obtain
		\beq
		A(t)&=&\frac{x^{|v_0|}}{1-xt}+\sum_{w\in W}\left(\sum_{j\geq0}\M[B](v_j,w)x^{|v_0|+1-|w|}(xt)^j\right)F_w(x)\\
		&& +\frac{x^2t}{1-xt}A(t)+axA(t)+\frac{A(t)-A(0)}{t},
		\feq
		which is equivalent to
		\beq
		\left(1-\frac{x^2t}{1-xt}-ax-\frac{1}{t}\right)A(t)
		&=&\frac{x^{|v_0|}}{1-xt}+\sum_{w\in W}\left(\sum_{j\geq0}\M[B](v_j,w)x^{|v_0|+1-|w|}(xt)^j\right)F_w(x)\\
		&& -\frac{1}{t}A(0).
		\feq
		We apply the \emph{kernel method} (see, e.g., \cite{HouM}) to the last equation and take $t=t_0$, where $t_0$ satisfies $1-\frac{x^2t_0}{1-xt_0}-ax-\frac{1}{t_0}=0$, to obtain
		\beqn \label{l1112}
		F_{v_0}(x)=A(0)=\frac{x^{|v_0|}t_0}{1-xt_0}+t_0\sum_{w\in W}\left(\sum_{j\geq0}\M[B](v_j,w)x^{|v_0|+1-|w|}(xt_0)^j\right)F_w(x),
		\feqn
		with
		$$t_0=\frac{1+x-ax-\sqrt{(1+x-ax)(1-3x-ax)}}{2x(1+x-ax)}=1+ax+(a^2+1)x^2+\cdots.$$
	
By \eqref{l1112} and the assumed rationality of $\sum_{j\geq0}\M[B](v_j,w)x^j$, we have that $F_{v_0}(x)$ is a linear combination of $F_v(x)$ for $w \in W$ whose coefficients are rational in $x$ and $t_0$ with $W$ a finite set.  Upon considering $B(t)=\sum_{j\geq 1}F_{v_j}(x)t^{j-1}$ and finding $B(0)=F_{v_1}(x)$, one can show that $F_{v_1}(x)$ is a similar linear combination of $F_{v_0}(x)$ and the $F_w(x)$ with $w \in W$.  Upon substituting out this expression for $F_{v_1}(x)$ as needed, it is seen that the nodes of length $m$ meet the conditions of Theorem \ref{thg2}, but where the coefficients $c_j(x)$ are now rational in $x$ and $t_0$.  Thus, it follows that the generating function $F_1(x)=G_B(x)$ is rational in $x$ and $t_0$.
	\end{proof}
	
	In the following example, we elaborate on how to apply Theorem \ref{thres2} to Example \ref{impex_11}.
	
	\begin{example}\label{exT10aa2}
		It is easy to check  that Theorem \ref{thres2} applies to the case
		\beq
		B=\{1243, 1324, 1342, 1423, 1432, 2143, 2413, 2431, 3142, 4132\},
		\feq
		and hence $G_B(x)$ is rational in $x$ and $t_0=C(x)$, where $C(x)=\frac{1-\sqrt{1-4x}}{2x}=\sum_{n\geq 0}C_nx^n$ and $C_n=\frac{1}{n+1}\binom{2n}{n}$ is the $n$-th Catalan number.  The set of succession rules given by \eqref{exmp-rb-10} can be written as the system of equations
		\beq
		&&F_1(x)=x+F_{12}(x)+F_{21}(x),\\
		&&F_{12}(x)=x^2+xF_{21}(x)+F_{123}(x)+F_{132}(x),\\
		&&F_{21}(x)=x^2+xF_{21}(x)+F_{213}(x)+F_{321}(x),\\
		&&F_{132}(x)=x^3,\\
		&&F_{123}(x)=x^3+x^2F_{21}(x)+xF_{123}(x),\\
		&&F_{213}(x)=x^3+x^2F_{21}(x)+xF_{123}(x),\\
		&&F_{321}(x)=A(0).
		\feq
		In addition, from $\M[B](v_j,21)=\M[B](v_j,213)=1$ for all $j \geq0$, have
		$$A(0)=\frac{x^3t_0}{1-xt_0}+\frac{x^2t_0}{1-xt_0}F_{21}(x)+
		\frac{xt_0}{1-xt_0}F_{213}(x).$$
		Hence, by solving the above system, we obtain $G_B(x)=F_1(x)=x^3-1+C(x)$, as was shown in \cite[Lemma 4.13]{MSS0}.
		
		Similarly, one may verify that Theorem \ref{thres2} is applicable to the case
		\beq
		B=\{1234, 1243, 1324, 1342, 1423, 2134, 2314, 2341, 3124, 4123\}.
		\feq
		In particular, it yields $G_B(x)=x^3-1+C(x)$, as was also shown in \cite[Lemma 4.13]{MSS0}.
	\end{example}

\noindent \emph{Remark:} Note that we may allow for nodes in $W$ in the definition above to have length greater than $m$, provided it is required that $W$ be finite where there are no edges $(a,b)$ such that $a \in W$ with $|a|>m$ and $b=v_i$ for some $i \geq 2$.  The same can be said in the case when $\D[B]$ is almost path-directed.
	
	\subsection{Directed graph with $\alpha$-growing paths}
	For $\alpha\geq 1$, suppose that the set $V$ of nodes in $\D[B]$ for some $B$ can be partitioned into sets $V^{(j)}$ for $1 \leq j \leq \alpha$ and $W$, where $V^{(j)}:=\{v_i^{(j)}\  |\  i\geq 0\}$ and $W$ is a finite set with $W:=\{w_1,\ldots,w_\ell\}$ for some $\ell \geq 1$.  Then we will say that $V$ has \emph{$\alpha$-growing paths} if the following conditions are satisfied:
	\begin{itemize}
		\item $|v_i^{(1)}|=\cdots = |v_i^{(\alpha)}|=i+m$ for all $i \geq 0$ where $m \geq 1$ is fixed.
		\item The length of any node in $W$ is at most $m$.
		\item The edges that start with a member of $\cup_{j=1}^\alpha V^{(j)}$ are dictated by the following succession rules:
		\begin{align*}
			v_k^{(1)}\rightsquigarrow& v_0^{(1)},v_1^{(1)},\ldots,v_{k-1}^{(1)},(v_k^{(1)})^{r_1},(w_1)^{p_{1,1}^{(k)}},\ldots,(w_\ell)^{p_{1,\ell}^{(k)}},\\
			v_k^{(s)}\rightsquigarrow &v_0^{(s')},v_1^{(s')},\ldots,v_{k-1}^{(s')},(v_{k}^{(s')})^{r_s},(v_{k+1}^{(s')})^{r_{s,1}},(v_{k}^{(s)})^{r_{s,2}},\\
&(w_1)^{p_{s,1}^{(k)}},\ldots,(w_\ell)^{p_{s,\ell}^{(k)}},\quad 1\leq s'<s\leq \alpha-1,\\
			v_k^{(\alpha)}\rightsquigarrow &v_0^{(\alpha')},v_1^{(\alpha')},\ldots,v_{k-1}^{(\alpha')},(v_{k}^{(\alpha')})^{r_\alpha},(v_k^{(\alpha)})^{r_\alpha'},(v_{k+1}^{(1)})^{r_{\alpha,1}},\ldots,(v_{k+1}^{(\alpha)})^{r_{\alpha,\alpha}},\\
&(w_1)^{p_{\alpha,1}^{(k)}},\ldots,(w_\ell)^{p_{\alpha,\ell}^{(k)}},\quad
			1\leq \alpha'<\alpha.
		\end{align*}
\item All other edges start with a node in $W$.
	\end{itemize}
Here, the non-negative exponents $r_\alpha'$, $r_i,r_{\alpha,i}$ for $i \in [\alpha]$, $r_{i,1},r_{i,2}$ for $i \in [2,\alpha-1]$ and $p_{i,j}^{(k)}$ for $k \geq 0$, $1 \leq i \leq \alpha$ and $1 \leq j \leq \ell$ are all assumed to be fixed.  Further, the parameters $\alpha'$ and $s'$ in the penultimate condition above are also fixed with $s'$ depending upon $s>1$. For each $1\leq j \leq \alpha$, it is often the case that the node $v_i^{(j)}\in V^{(j)}$ for $i>0$ contains either the subword $m(m+1)\cdots(m+i)$ or $(m+i)(m+i-1)\cdots m$.
	
	We now consider two examples of sets of patterns whose corresponding directed graphs have $\alpha$-growing paths.
	
\begin{figure}[htp]
		\begin{tikzpicture}[->]
			\node (1) at (10,15) {1};
			\node (21) at (8,13.5) {21};
			\node (12) at (11,13.5) {12};
			\node (c3) at (4,11.5) {$c_3$};
			\node (b3) at (6,11.5) {$b_3$};
			\node (a3) at (8,11.5) {$a_3$};
			\node (132) at (11,10.5) {132};
			\node (312) at (12,11.5) {312};
			\node (c4) at (3.5,9) {$c_4$};
			\node (b4) at (5.5,10) {$b_4$};
			\node (a4) at (7.5,10) {$a_4$};
			\node (c5) at (3.5,7) {$c_5$};
			\node (b5) at (4.7,8.3) {$b_5$};
			\node (a5) at (6.7,8.3) {$a_5$};
			\node (c6) at (3.5,5) {$\vdots$};
			\node (b6) at (4.2,6.3) {$\vdots$};
			\node (a6) at (6.2,6.3) {$\vdots$};
			
			\path (1) edge  node[above] {} (12);
			\path (1) edge  node[above,right] {} (21);
			\path (12) edge  node[above,left] {} (132);
			\path (12) edge  node[above, right] {} (312);
			\path (12) edge [loop above] node {} (12);
			\path (21) edge  node[above, right] {} (c3);
			\path (21) edge  node[above, right] {} (b3);
			\path (21) edge  node[above, right] {} (a3);
			\path (132) edge [loop above] node {(2)} (132);
			\path (a3) edge [loop above] node {} (a3);
			\path (b3) edge [loop above] node {} (b3);
			\path (b3) edge  node[above, right] {} (a3);
			\path (b3) edge  node[above, right] {} (132);
			\path (c3) edge  node[above, right] {} (a4);
			\path (c3) edge  node[above, right] {} (b4);
			\path (c3) edge  node[above, right] {} (c4);
			\path (a4) edge [loop above] node {} (a4);
			\path (b4) edge [loop above] node {} (b4);
			\path (a4) edge  node[above, right] {} (a3);
			\path (b4) edge  node[above, right] {} (a4);
			\path (b4) edge  node[above, right] {} (a3);
			\path (b4) edge  node[above, right] {} (132);
			\path (c4) edge  node[above, right] {} (a4);
			\path (c4) edge  node[above, right] {} (a5);
			\path (c4) edge  node[above, right] {} (b5);
			\path (c4) edge  node[above, right] {} (c5);
			\path (a5) edge [loop above] node {} (a5);
			\path (b5) edge [loop above] node {} (b5);
			\path (c5) edge  node[above, right] {} (a4);
			\path (c5) edge  node[above, right] {} (a5);
			\path (b5) edge  node[above, right] {} (a4);
			\path (b5) edge  node[above, right] {} (a5);
			\path (b5) edge  node[above, right] {} (a3);
			\path (b5) edge  node[above, right] {} (132);
			\path (a5) edge[out=0,in=-90]  node[above, right] {} (a3);
			\path (a5) edge  node[above, right] {} (a4);
			\path (c5) edge  node[above, right] {} (c6);
			\path (c5) edge  node[above, right] {} (b6);
			\path (c5) edge  node[above, right] {} (a6);
\path (c3) edge[out=0,in=-60]  node[above, right] {} (a3);
\path (c4) edge[out=0,in=90]  node[above, right] {} (a3);
\path (c5) edge[out=0,in=-90]  node[above, right] {} (a3);
		\end{tikzpicture}		
		\caption{The label on the loop indicates that is repeated.}\label{figEX13a1}
	\end{figure}
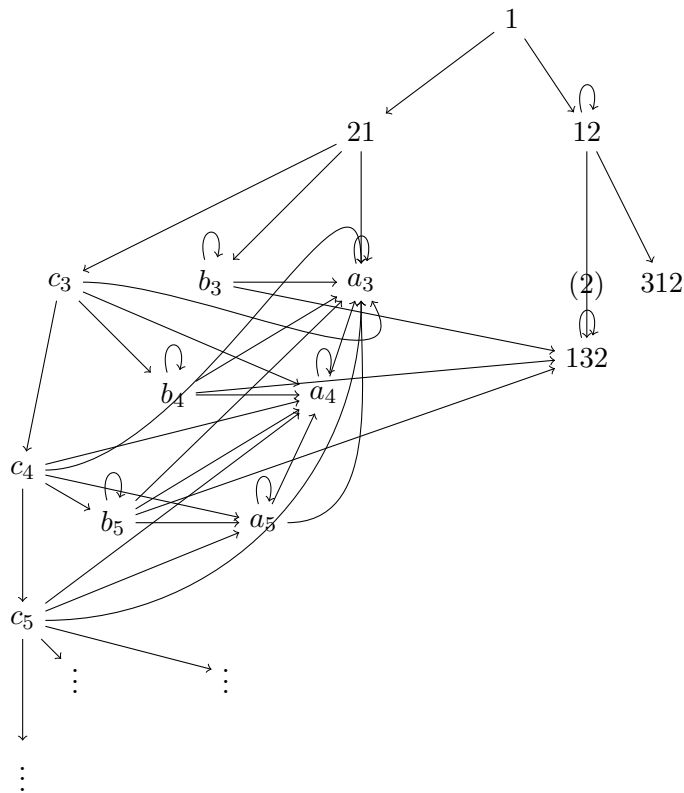			

	\begin{example}\label{exT1T9}
		Let $B=\{1324,1423,2143,2413,3124,3142,3412,4132,4213,4231,4312\}.$ Then, by Algorithm \ref{alg1}, $\R[B]$ is given by
		\beq
		\R[B]&=&\{1\rightsquigarrow12,21\}\\
		&&\cup\{12\rightsquigarrow12,132,312\}\\
		&&\cup \{21\rightsquigarrow213,231,321\}\\
		&&\cup \{132\rightsquigarrow132^2\}\\
		&&\cup \{a_k\rightsquigarrow a_4,\ldots,a_k,213, \quad k\geq3\}\\
		&&\cup \{b_k\rightsquigarrow a_4,\ldots,a_k,b_k,132,213, \quad k\geq3\}\\
		&&\cup \{c_k\rightsquigarrow a_4,\ldots,a_k,a_{k+1},b_{k+1},c_{k+1},213, \quad k\geq3\},
		\feq
		where
		\beq
		a_k&=&(k-1)(k-2)k(k-3)(k-4)\cdots1,\\
		b_k&=&(k-1)k(k-2)(k-3)\cdots1,\\
		c_k&=&k(k-1)\cdots1.
		\feq
		Here, $\D[B]$ (see Figure \ref{figEX13a1}) has $3$-growing paths.
	\end{example}
	
	\begin{example}
		Let $B=\{1243,1324,1342,1423,1432,2143,2413,2431,3142,3412,4132\}$. $\R[B]$ is given by
		\beq
		\R[B]&=&\{1\rightsquigarrow12,21\}\\
		&&\cup \{12\rightsquigarrow123,132,312\}\\
		&&\cup \{21\rightsquigarrow213,231,321\}\\
		&&\cup \{123\rightsquigarrow123^2\}\\
		&&\cup \{213\rightsquigarrow123,213 \}\\
		&&\cup \{312\rightsquigarrow123,312\}\\
		&&\cup \{a_k\rightsquigarrow a_3,\ldots,a_{k-1},a_k^2,213, \quad k\geq3\}\\
		&&\cup \{b_k\rightsquigarrow a_3,\ldots,a_k,a_{k+1},b_{k+1},213, \quad k\geq3\},
		\feq
		where
		\beq
		a_k&=&(k-1)k(k-2)(k-3)\cdots1,\\
		b_k&=&k(k-1)\cdots1.
		\feq
		Here, $\D[B]$ (see Figure \ref{figEX13a2}) has $2$-growing paths.
		\begin{figure}[htp]
			\begin{tikzpicture}[->]
				\node (1) at (10,15) {1};
				\node (21) at (8,13.5) {21};
				\node (12) at (11,13.5) {12};
				\node (312) at (10,12) {312};
				\node (123) at (11,11.5) {123};
				\node (132) at (12,12) {132};
				\node (b3) at (5.5,12) {$b_3$};
				\node (a3) at (6.5,12) {$a_3$};
				\node (213) at (8,12) {213};
				\node (b4) at (5,10) {$b_4$};
				\node (a4) at (7,10) {$a_4$};
				\node (b5) at (5,8.5) {$b_5$};
				\node (a5) at (7,8.5) {$a_5$};
				\node (b6) at (5,7) {$\vdots$};
				\node (a6) at (7,7) {$\vdots$};
				
				\path (123) edge [loop above] node {2} (123);
				\path (312) edge [loop above] node {} (312);
				\path (213) edge [loop above] node {} (213);
				\path (a3) edge [loop above] node {2} (a3);
				\path (a4) edge [loop above] node {2} (a4);
				\path (a5) edge [loop above] node {2} (a5);
				\path (1) edge  node[above, right] {} (12);
				\path (1) edge  node[above, right] {} (21);
				\path (12) edge  node[above, right] {} (123);
				\path (12) edge  node[above, right] {} (132);
				\path (12) edge  node[above, right] {} (312);
				\path (21) edge  node[above, right] {} (a3);
				\path (21) edge  node[above, right] {} (b3);
				\path (21) edge  node[above, right] {} (213);
				\path (213) edge[out=-25,in=-45]  node[above, right] {} (123);
				\path (a3) edge  node[above, right] {} (213);
				\path (b3) edge[out=-45,in=-45]  node[above, right] {} (213);
				\path (b3) edge  node[above, right] {} (a3);
				\path (b3) edge  node[above, right] {} (a4);
				\path (b3) edge  node[above, right] {} (b4);
				\path (b4) edge  node[above, right] {} (b5);
				\path (b4) edge  node[above, right] {} (a5);
				\path (b4) edge  node[above, right] {} (a4);
				\path (b4) edge  node[above, right] {} (213);
				\path (b4) edge  node[above, right] {} (a3);
				\path (a4) edge  node[above, right] {} (213);
				\path (a4) edge  node[above, right] {} (a3);
				\path (a5) edge[out=-15,in=-35]  node[above, right] {} (213);
				\path (a5) edge[out=20,in=-45]  node[above, right] {} (a3);
				\path (a5) edge  node[above, right] {} (a4);
				\path (b5) edge  node[above, right] {} (213);
				\path (b5) edge  node[above, right] {} (a3);
				\path (b5) edge  node[above, right] {} (a4);
				\path (b5) edge  node[above, right] {} (a5);
				\path (b5) edge  node[above, right] {} (a6);
				\path (b5) edge  node[above, right] {} (b6);
				\path (312) edge  node[above, right] {} (123);
			\end{tikzpicture}		
			\caption{The labels on four of the loops indicate their repetition as edges in the graph.}\label{figEX13a2}
		\end{figure}
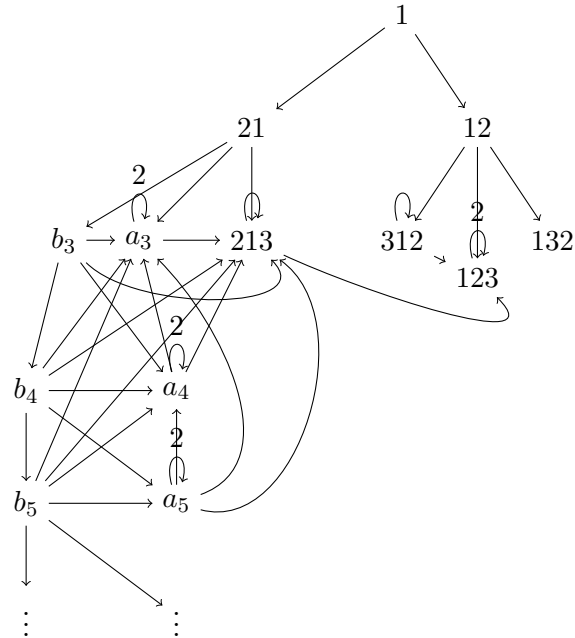		
	\end{example}
	
	If $B$ is a set of patterns whose directed graph $\D[B]$ has $\alpha$-growing paths, we may apply Algorithm \ref{alg1} and the following result to deduce the rationality of $G_B(x)$.
	
	\begin{theorem}\label{thresc}
		Let $1\notin B \subset S$ be any set of patterns whose corresponding $\D[B]$ has $\alpha$-growing paths with $V^{(1)},\ldots,V^{(\alpha)},W$ as described above. If the generating function given by $\sum_{j\geq0}\M[B](v_j^{(i)},w)x^j$ is rational for each $i \in [\alpha]$ and $w\in W$, then $G_B(x)$ is rational.
	\end{theorem}
	\begin{proof}
We prove the statement in the case when $|\alpha|\geq 3$, as the adjustments required for the $\alpha=1,2$ cases will be apparent.  Note that $\{v_1^{(1)},v_1^{(2)},\ldots,v_1^{(\alpha)}\}$ comprises the set of all nodes in $V$ of length $m+1$.  By Theorem \ref{thg2}, it suffices to show that $F_{v_1^{(i)}}(x)$ for each $i \in [\alpha]$ is a rational linear combination of the $F_v(x)$ for $v \in V$ with $|v|\leq m$ wherein the corresponding coefficients all vanish at $x=0$.  In order to aid in doing so, we define the generating function $A_i(t)=\sum_{k\geq0}F_{v_k^{(i)}}(x)t^k$ for $1 \leq i \leq \alpha$.  From the succession rules, we have
\begin{equation}\label{thresce1}
\left(1-r_1x-\frac{x^2t}{1-xt}\right)A_1(t)=\frac{x^m}{1-xt}+\sum_{w\in W}x^{m+1-|w|}\left(\sum_{k\geq0}\M[B](v_k^{(1)},w)(xt)^k\right)F_w(x),
\end{equation}
\begin{align}
\left(1-r_{s,2}x\right)A_s(t)&=\frac{x^m}{1-xt}+\left(\frac{x^2t}{1-xt}+\frac{r_{s,1}}{t}+r_sx\right)A_{s'}(t)-r_{s,1}\frac{A_{s'}(0)}{t}\notag\\
&\quad+\sum_{w\in W}x^{m+1-|w|}\left(\sum_{k\geq0}\M[B](v_k^{(s)},w)(xt)^k\right)F_w(x), \quad 2 \leq s \leq \alpha-1, \label{thresce2}
\end{align}
and
\begin{align}
\left(1-r_\alpha'x\right)A_\alpha(t)&=\frac{x^m}{1-xt}+\left(\frac{x^2t}{1-xt}+r_\alpha x\right)A_{\alpha'}(t)+\sum_{i=1}^\alpha r_{\alpha,i}\left(\frac{A_i(t)-A_i(0)}{t}\right)\notag\\
&\quad+\sum_{w\in W}x^{m+1-|w|}\left(\sum_{k\geq0}\M[B](v_k^{(\alpha)},w)(xt)^k\right)F_w(x). \label{thresce3}
\end{align}
\indent By \eqref{thresce1}, we have that $A_1(t)$ is a linear combination of the $F_w(x)$ for $w \in W$ whose coefficients $c_j(x,t)$ are rational in $x$ and $t$ and satisfy $c_j(0,t)=0$ for all $j$ and each fixed $t$.  Then by an induction argument using \eqref{thresce2}, we have that $A_s(t)$ for $2 \leq s \leq \alpha-1$ also admits of such a form.  Note that $r_{\alpha,i}\in\{0,1\}$ for $i \in [\alpha]$ since there is at most one way to produce a certain offspring of length $k+1$ from a parent of length $k$.  If $r_{\alpha,\alpha}=0$, then \eqref{thresce3} implies, like in the prior cases, that $A_\alpha(t)$ admits of this form too.  If $r_{\alpha,\alpha}=1$, then taking $t=1$ in \eqref{thresce3}, and solving for $A_\alpha(0)$, implies $A_\alpha(0)$ has the desired form and hence $A_\alpha(t)$ does as well.  Thus, for each $i \in [\alpha]$, we have in particular that $F_{v_1^{(i)}}(x)=\frac{A_i(t)-A_i(0)}{t}\mid_{t=0}$ is a rational linear combination of the $F_v(x)$ with $|v|\leq m$ of the form stated above, which completes the proof.
\end{proof}
	
	\begin{example}
		As shown in \cite{MSS0} using different techniques, there are exactly $10$ sets of patterns of size 11 consisting of members of $S_4$ where the FinLabel algorithm fails to terminate in a  finite number of iterations, and they are given by
		\begin{align*}
			B_1=&\{1324,1423,2143,2413,3124,3142,3412,4132,4213,4231,4312\},\\
			B_2=&\{1324,1423,2143,2413,3124,3142,4123,4132,4213,4231,4312\},\\
			B_3=&\{1324,1423,2143,2413,3124,3142,3412,4123,4132,4213,4231\},\\
			B_4=&\{1324,1423,2143,3124,3142,3412,4123,4132,4213,4231,4312\},\\
			B_5=&\{1324,1423,2143,2413,3124,3142,3412,4123,4132,4231,4312\},\\
			B_6=&\{1243,1324,1342,1423,1432,2143,2413,3142,3412,4132,4231\},\\
			B_7=&\{1324,1423,2143,2413,3124,3142,3412,4123,4132,4213,4312\},\\
			B_8=&\{1324,1423,2413,3124,3142,3412,4123,4132,4213,4231,4312\},\\
			B_9=&\{1243,1324,1342,1423,1432,2143,2413,2431,3142,3412,4132\},\\
			B_{10}=&\{1234,1243,1324,1342,1423,2134,2314,2341,3124,3412,4123\}.
		\end{align*}

		The set $B_1$ was discussed in Example \ref{exT1T9}, where $\D[B_1]$ was seen to have $3$-growing paths with $V^{(1)}=\{(k-1)(k-2)k(k-3)(k-4)\cdots 1\mid k \geq 4\}$, $V^{(2)}=\{(k-1)k(k-2)(k-3)\cdots 1\mid k \geq 4\}$, $V^{(3)}=\{k(k-1)\cdots 1\mid k \geq 4\}$ and $W=\{1,12,21,132,213,231,312,321\}$. Note that
		$\sum_{v\in W}\sum_{j\geq0}\M[B](v_j^{(s)},v)x^j$ equals $(1-x)^{-1}$ when $s=1,3$ and $2(1-x)^{-1}$ when $s=2$. Hence, Theorem \ref{thresc} implies $G_{B_1}(x)$ is a rational generating function. As noted in the following table, all other cases also have $\alpha$-growing paths. One can readily show that Theorem \ref{thresc} is applicable in each of these cases as well. Additionally, following the same process as the one given in Example \ref{ex12331221543}, one can calculate $G_{B_s}(x)$ in all cases, which we omit here for the sake of brevity.\\
		
		\begin{center}
			\begin{tabular}{ |c|c|}
				\hline
				$B_s$& $\alpha$-growing paths \\
				\hline
				$B_1$,$B_2$, $B_3$, $B_4$, $B_5$, $B_{10}$  & 3  \\
				$B_6$, $B_7$ & 4 \\
				$B_8$, $B_9$ & 2  \\
				\hline
			\end{tabular}
		\end{center}
	\end{example}
	
	We conclude by mentioning some further applications of the preceding results. Algorithm \ref{alg1} together with Theorems \ref{thres1} and \ref{thresc} have been applied to many sets $B$ consisting of members of $S_4$ where $3\leq |B|\leq 12$. In our study, we have shown that there are $48$ sets of patterns of size ten for which the FinLabel algorithm fails to find $G_B(x)$ in a finite number of iterations. From these cases, exactly $10$ (resp. $19$, $10$, $2$ and $3$) cases have directed graphs with $2$-growing (resp. $3$-, $4$-, $5$- and $6$-) paths. In other words, with the exception of four cases, Theorem \ref{thresc} (and its proof) provides the solution to the problem of finding the generating function $G_B(x)$ when $|B|=10$. Interestingly enough, Theorem \ref{thres2} yields $G_B(x)$ for the four remaining cases, two of which were treated in Example \ref{exT10aa2}.

	
\end{document}